\documentclass[11pt,reqno]{amsart}

\usepackage{amssymb,amsmath,graphicx,amsfonts,euscript}
\usepackage{color}

\newtheorem{thm}{Theorem}[section]

\newtheorem{rem}[thm]{Remark}

\newtheorem{lemma}[thm]{Lemma}

\newcommand{\p}{\partial}

\newcommand{\na}{\nabla}

\newcommand{\ddt}{\frac{d}{dt}}
\newcommand{\R}{\mathbb{R}}

\numberwithin{equation}{section}
\begin{document}
\title[Global well-posedness for the  2D   Boussinesq  equations]{Global well-posedness for the  2D   Boussinesq  equations with a velocity damping term}

\author[  R. Wan]{ Renhui Wan}
\address{ School of Mathematical Sciences, Nanjing Normal University, Nanjing 210023, China}

\email{rhwanmath@163.com}

\vskip .2in
\begin{abstract}
In this paper, we prove global well-posedness of smooth solutions to the two-dimensional incompressible Boussinesq equations with only a velocity damping term when the initial data is close to an nontrivial equilibrium state $(0,x_2)$. As a by-product, under this equilibrium state, our result  gives a positive answer to the question proposed by \cite{ACWX} (see P.3597).
\end{abstract}

\maketitle

\vskip .2in
\section{Introduction}
\label{s1}
\vskip .3in
In this paper, we investigate the global existence and uniqueness of solutions to the two-dimensional (2D) incompressible   Boussinesq equations with a velocity damping term, namely,
\begin{equation} \label{ODBE}
\left\{
\begin{array}{l}
\partial_t u + u\cdot\nabla u +  \nu u +\nabla p = \Psi e_2,\  (t,x)\in \R^+\times \R^2 \\
\partial_t \Psi + u\cdot\nabla \Psi  =0,\\
{\rm div} u=0, \\
u|_{t=0} =u_0(x), \quad \Psi|_{t=0}=\Psi_0(x),
\end{array}
\right.
\end{equation}
where $u=(u_1,u_2)$ stand for the 2D velocity field, $p$ the pressure and     $\Psi$  the temperature in thermal convection or the density in geophysical flows, $\nu>0$ is a parameter  and $e_2=(0,1)$ is the unit vector in the vertical direction. $u_0$ satisfies ${\rm div}u_0=0$.
\vskip .1in
Many geophysical flows such as atmospheric fronts and
ocean circulations (see, e.g., \cite{CD99,Gill,Majda,Ped}) can be modeled
by the Boussinesq equations.  Mathematically we regard the 2D Boussinesq equations
as a lower-dimensional model of the 3D hydrodynamics equations.
 It is an open question whether the solutions to the inviscid case $(\nu=0)$ exist for all time or blow-up in a finite time.
 Indeed, adding only a velocity damping term seems not helpful to get the global solution even under the assumption that the initial data is small enough.
 \vskip .1in
 When adding the temperature damping term $\eta \Psi$ $(\eta>0)$ to (\ref{ODBE}), that is,
\begin{equation} \label{OODBE}
\left\{
\begin{array}{l}
\partial_t u + u\cdot\nabla u +  \nu u +\nabla p = \Psi e_2,\  (t,x)\in \R^+\times \R^2 \\
\partial_t \Psi + u\cdot\nabla \Psi+\eta\Psi  =0,\\
{\rm div} u=0,
\end{array}
\right.
\end{equation}
 Adhikar et al. \cite{ACWX} proved global well-posedness to (\ref{OODBE}) with the initial data satisfying
$$\|\nabla u_0 \|_{\dot{B}_{\infty,1}^0}<\min\{\frac{\nu}{2C_0},\frac{\eta}{C_0}\},\ \|\nabla \Psi_0 \|_{\dot{B}_{\infty,1}^0}<\frac{\nu}{2C_0}\|\nabla u_0 \|_{\dot{B}_{\infty,1}^0},$$
where $\dot{B}_{\infty,1}^0$ is the Besov space (see \cite{BCD} for the definition),  while they proposed that  global well-posedness for (\ref{ODBE})  is an open question  (see P. 3597 in that paper). Later, Wan \cite{W} obtained global solutions to (\ref{OODBE}) with large  initial velocity data by exploiting a new decomposition technique which is splitting the damped Navier-Stokes equations from (\ref{OODBE}). In fact, in \cite{W}, the initial data satisfies
\begin{equation*}
\begin{aligned}
\frac{C_0}{\sqrt{\nu\eta}}\|\Psi_0\|_{H^{m-1}}&
\exp\left\{C(\frac{1}{\nu}+\eta)\|u_0\|_{H^m}\exp\{\frac{C}{\nu}\|\|\Omega_0\|_{L^2\cap L^\infty}A(\nu,u_0,\Omega_0)\}\right\}\\
<&\min\{\nu,\ \frac{1}{\nu}\}\ (m>3)
\end{aligned}
\end{equation*}
for some constants $C>0$ and $C_0>0$, where $\Omega_0\stackrel{\rm def}{=}\nabla \times u_0$ and $A(\nu,u_0,\Omega_0)\stackrel{\rm def}{=}\ln (e+\frac{\|u_0\|_{H^m}}{\nu})\exp\left\{\frac{C\|\Omega_0\|_{L^2\cap L^\infty}}{\nu}\right\}$.  We refer \cite{WXY} and references therein for the related works.
\vskip .1in
Inspired by the studies of the MHD equations ( see, e.g., \cite{RWXZ,Wu15,Wu17}),
the contribution of this article is the global existence and uniqueness of  solutions of (\ref{ODBE}) with sufficiently smooth initial data $(u_0,\Psi_0)$ close to the equilibrium state $(0,x_2)$.
In fact, using $\Psi=\theta+x_2$,  we return to  seek the solutions of
the following system
\begin{equation} \label{DBE1}
\left\{
\begin{array}{l}
\partial_t \omega + u\cdot\nabla \omega +  \nu\omega  = \p_1\theta,  \\
\partial_t \theta + u\cdot\nabla \theta  =-u_2,\\
u=\nabla^\bot\Lambda^{-2}\omega, \\
\omega|_{t=0} =\omega_0(x), \quad \theta|_{t=0}=\theta_0(x),
\end{array}
\right.
\end{equation}
where $\na^\bot\stackrel{\rm def}{=} (\p_2,-\p_1)$ and
$\omega\stackrel{\rm def}{=}\p_1u_2-\p_2u_1$ is the vorticity of the velocity $u$. When $\nu=0$,  Elgindi-Widmayer \cite{EW} proved the long time existence of (\ref{DBE1}), that is, the lifespan of the associated solutions  is $\epsilon^{-\frac{4}{3}}$ if the initial data is of size $\epsilon$. We point out that   the initial data in $W^{s,1}(\R^2)$ is very crucial in the decay estimate of the semi-group $e^{\mathcal{R}_1t}$.  Later, Wan-Chen \cite{WC} obtained global well-posedness under the initial data near a  nontrivial equilibrium $(0,\kappa x_2)$ ($\kappa$ large enough), which is consistent with the corresponding work of \cite{EW} if $\kappa=1$. Recently, by establishing   some Strichartz type estimates of the semi-group $e^{\mathcal{R}_1t}$, Wan \cite{Wan17} obtained the same result as \cite{EW} without assuming that the initial data is in $W^{s,1}(\R^2)$.
\vskip .1in
For some convenience, we set $\nu=1$, that is, we consider
\begin{equation} \label{DBE}
\left\{
\begin{array}{l}
\partial_t \omega + u\cdot\nabla \omega +  \omega  = \p_1\theta,  \\
\partial_t \theta + u\cdot\nabla \theta  =-u_2,\\
u=\nabla^\bot\Lambda^{-2}\omega, \\
\omega|_{t=0} =\omega_0(x), \quad \theta|_{t=0}=\theta_0(x)
\end{array}
\right.
\end{equation}
 in the present work and use the  denotations below:
\begin{equation*}
\begin{aligned}
A(T)\stackrel{\rm def}{=}&\|\omega\|_{L^\infty_T(\mathbb{H}^s)}^2
+\|\theta\|_{L^\infty_T(\mathcal{H}^{s+1})}^2,\\
A_1(T)\stackrel{\rm def}{=}& \|\omega\|_{L^2_T(\mathbb{H}^s)}^2
+\|\p_1\theta\|_{L^2_T(\mathbb{H}^{s})}^2,\\
A(0)\stackrel{\rm def}{=}&\|\omega_0\|_{\mathbb{H}^s}^2
+\|\theta_0\|_{\mathcal{H}^{s+1}}^2.
\end{aligned}
\end{equation*}
where the definitions of the  function spaces $\mathcal{H}^s(\R^2)$ and $\mathbb{H}^s(\R^2)$ are given in section \ref{s2}.
 Now, we state the main result.
\begin{thm}\label{t1.1}
Let  $s\ge 5$.  Assume that
$$\omega_0\in \mathbb{H}^s(\R^2),\ \  \theta_0\in \mathcal{H}^{s+1}(\R^2).$$
There exists a positive constant   $C'$ such that    if
\begin{equation}\label{small}
C'A(0)<1
\end{equation}
then  system (\ref{DBE}) has  a unique global solution
$(\omega,\theta)$ satisfying
$$A(T)+A_1(T)\le C'A(0),$$
for all $T>0$.
\end{thm}
\begin{rem}\label{r1}
We shall point out that the solutions in Sobolev space do not grow over time, which is very different from the work \cite{Wu15} on the 2D MHD system with a velocity damping term, where the Sobolev norm of the solutions may grow over time, see P.2631 in that paper. Furthermore, our result gives a positive answer to the question proposed in \cite{ACWX} under the initial data near equilibrium state $(0,x_2)$.
\end{rem}
Let us making the following comments concerning this theorem:
\begin{enumerate}
\item[$\bullet$] By the standard energy method, we need to bound the integral $\int_0^T \|\p_2u_2\|_{L^\infty}dt$, which is very difficult, even if one applies the techniques in the following sections.  By utilising  $-u_2=\p_t\theta+u\cdot\na\theta$ and integrating by parts, the estimate of this integral is reduced to the estimate of $\mathfrak{K}(T)$ (see (\ref{K}) for the definition).

\item[$\bullet$] To get the estimate of $\mathfrak{K}(T)$, we shall estimate the integral
$\int_0^T\|u_2\|_{L^\infty}^\frac{4}{3}dt$, which strongly relies on the diagonalization process (section \ref{s5}) and the energy estimate II (section \ref{s6}). In addition, we shall use some unexpected techniques like (\ref{tech1}).

\item[$\bullet$] The unnatural condition in (\ref{small}) is that   the initial data satisfies $\omega_0\in \dot{H}^{-2}(\R^2)$ and $\theta_0\in \dot{H}^{-1}(\R^2)$. In fact,  these conditions play a very important role in the energy estimate II and the estimate of $\int_0^T\|u_2\|_{L^\infty}^\frac{4}{3}dt$.
\end{enumerate}
\vskip .1in
The present paper is structured as follows:\\
 In section \ref{s2},  we provide some definitions of  spaces and  several lemmas. Section \ref{s3} devotes to obtaining the  estimate of $E(T)+E_1(T)$ (see Lemma \ref{l3.1} for the definition).  Section \ref{s4} bounds the estimate of $\int_0^T\int \p_2u_2(\p_2^{s+1}\theta)^2 dxdt$. Section \ref{s5} provides the  estimate of $\int_0^T\|u_2\|_{L^\infty}^\frac{4}{3}dt$. In section \ref{s6}, we give the estimate of $E_2(T)$ (see (\ref{901}) for the definition). In the last section, we prove Theorem \ref{t1.1}.
\vskip .1in
Let us complete this section by describing the notations we shall use in this paper.\\
{\bf Notations} For two operators $A$ and $B$, we denote by $[A,B]=AB-BA$ the commutator between $A$ and $B$.  In some places of this paper,  we may use $L^p$, $\dot{H}^s$ and  $H^s$  to stand for  $L^p(\R^2)$, $\dot{H}^s(\R^2)$ and $H^s(\R^2)$, respectively.
$a\thickapprox b$  means $\mathfrak{C}^{-1}b\le a\le \mathfrak{C}b$ for some positive constant $\mathfrak{C}$. $\langle t\rangle$ means $1+t$.
The uniform constant $C$ may be different on different lines.   We use  $\|f\|_{L^p}$ to denote
the $L^p(\R^2)$ norm of $f$, and use $L^p_T(X)=L^p([0,T];X)$.
 We shall denote by $(a|b)$  the $L^2$ inner product
of $a$ and $b$, and
$$(a|b)_{\dot{H}^s}\stackrel{\rm def}{=}(\Lambda^s a|\Lambda^s b),
 \ {\rm and}\ \ (a|b)_{\dot{H}^m}\stackrel{\rm def}{=} (\p^m a|\p^mb)\  (m {\rm \ is\  an\  integer}),
$$
$$(a|b)_{H^s}\stackrel{\rm def}{=}(a|b)+(a|b)_{\dot{H}^s}.$$
\vskip .2in
\section{Preliminaries}
\label{s2}
\vskip .3in
In this section, we provide some essential  definitions,  propositions and lemmas.
\vskip .1in
The fractional Laplacian operator $\Lambda^\alpha=(-\Delta)^\frac{\alpha}{2}$  is defined through the Fourier transform, namely,
$$\widehat{\Lambda^\alpha f}(\xi)\stackrel{\rm def}{=}|\xi|^\alpha \widehat{f}(\xi),$$
where the Fourier transform is given by
$$\widehat{f}(\xi)\stackrel{\rm def}{=}\int_{\R^2}e^{-ix\cdot\xi}f(x)dx.$$
Sometimes, we also use $\mathcal{F}[f]$ to stand for the Fourier transform of $f$.
\vskip.1in
The norms $\|f\|_{\dot{H}^{\sigma_1}(\R^2)}$ ($\sigma_1\in \R$) and $\|f\|_{H^{\sigma_2}(\R^2)}$ $(\sigma_2>0)$ can be defined by
$$\|f\|_{\dot{H}^{\sigma_1}(\R^2)}\stackrel{\rm def}{=}\|\Lambda^{\sigma_1} f\|_{L^2(\R^2)}$$
and
$$
\|f\|_{H^{\sigma_2}(\R^2)}\stackrel{\rm def}{=}\|f\|_{L^2(\R^2)}+\|\Lambda^{\sigma_2} f\|_{L^2(\R^2)}.
$$
Especially, provided that  $\sigma_1\in\R$ and  $\sigma_2>0$ are integers,  we also have
$$\|f\|_{\dot{H}^{\sigma_1}(\R^2)}\stackrel{\rm def}{=}\|\p^{\sigma_1}f\|_{L^2(\R^2)}$$
$$\|f\|_{H^{\sigma_2}(\R^2)}\stackrel{\rm def}{=}\|f\|_{L^2(\R^2)}+\|\p^{\sigma_2} f\|_{L^2(\R^2)}.$$
For some convenience, we introduce two new functional spaces as follows:
$$\mathcal{H}^s(\R^2)\stackrel{\rm def}{=}\dot{H}^{-1}(\R^2)\cap \dot{H}^{s}(\R^2),$$
$$\mathbb{H}^s(\R^2)\stackrel{\rm def}{=}\dot{H}^{-2}(\R^2)\cap \dot{H}^{s}(\R^2).$$
\begin{lemma}($i$)\cite{KP}
Let $s>0$, $1\le p,r\le \infty,$ then
\begin{equation}\label{kp}
\|\Lambda^s( fg)\|_{L^p(\R^2)}\le C\left\{\|f\|_{L^{p_{1}}(\R^2)}
\|\Lambda^s g\|_{L^{p_2}(\R^2)}
+\|g\|_{L^{r_{1}}(\R^2)}\|\Lambda^sf\|_{L^{r_2}(\R^2)}\right\},
\end{equation}
where $1\le p_{1},r_{1}\le \infty$ such that $\frac{1}{p}=\frac{1}{p_{1}}+\frac{1}{p_{2}}=\frac{1}{r_{1}}+\frac{1}{r_{2}}$.\\
($ii$)\cite{KPV} Let $s>0$, and $1<p<\infty$,  then
\begin{equation}\label{kpv}
\|[\Lambda^s,f]g\|_{L^p(\R^2)}\le C\left\{\|\nabla f\|_{L^{p_1}(\R^2)}\|\Lambda^{s-1}g\|_{L^{p_2}(\R^2)}
+\|\Lambda^{s}f\|_{L^{p_3}(\R^2)}\|g\|_{L^{p_4}(\R^2)}\right\}
\end{equation}
where $1<p_2,p_3<\infty$ such that $\frac{1}{p}=\frac{1}{p_1}+\frac{1}{p_2}=\frac{1}{p_3}+\frac{1}{p_4}$.
\end{lemma}

\begin{lemma}\label{l2.1}
If the functions $f_i$ $(i=1,2,3,4,5)$ satisfy that
$$\p_1f_j\in L^2_T(H^1(\R^2))\ f_j\in L^\infty_T(H^1(\R^2)),\ j=1,2,$$
$$f_3\in L^p_T(L^\infty(\R^2)),\ f_4\in L^q_T(L^2(\R^2)),\ f_5\in L^\infty_T(L^2(\R^2)),$$
where $\frac{1}{p}+\frac{1}{q}=\frac{1}{2}$, $(p,q)\in [2,\infty]$,
then there holds
\begin{equation*}
\begin{aligned}
&\ \ \ |\int_0^T\int_{\R^2}f_1(x)f_2(x)f_3(x)f_4(x)f_5(x)dxdt|\\
\le&
C\|(\p_1f_1,\p_1f_2)\|_{L^2_T(H^1)}
\|(f_1,f_2)\|_{L^\infty_T(H^1)}\|f_3\|_{L^p_T(L^\infty)}
\|f_4\|_{L^q_T(L^2)}\|f_5\|_{L^\infty_T(L^2)}.
\end{aligned}
\end{equation*}
\end{lemma}
\begin{proof}
By using
$$\|f\|_{L^\infty_{x_1}}\le C\|f\|_{L^2_{x_1}}^\frac{1}{2}\|\p_1f\|_{L^2_{x_1}}^\frac{1}{2},$$
Minkowski's inequality for integrals
and the embedding relation $H^1(\R)\hookrightarrow L^\infty(\R)$, we have
\begin{equation*}
\begin{aligned}
&\ \ \ |\int_{\R^2}f_1(x)f_2(x)f_3(x)f_4(x)f_5(x)dx |\\
\le& \ |\int_{x_2\in \R}\|f_4\|_{L^2_{x_1}}\|f_5\|_{L^2_{x_1}}\prod_{i=1}^3\|f_i\|_{L^\infty_{x_1}}d x_2|\\
\le& \ C\|f_3\|_{L^\infty(\R^2)}\|f_4\|_{L^2(\R^2)}\|f_5\|_{L^2(\R^2)}
\left\|\prod_{i=1}^2\|f_i\|_{L^2_{x_1}}^\frac{1}{2}
\|\p_1f_i\|_{L^2_{x_1}}^\frac{1}{2}
\right\|_{L^\infty_{x_2}}\\
\le&\ C \|f_3\|_{L^\infty(\R^2)}
\|f_4\|_{L^2(\R^2)}\|f_5\|_{L^2(\R^2)}\\
&\times(\|f_1\|_{H^1(\R^2)}
+\|f_2\|_{H^1(\R^2)})(\|\p_1f_1\|_{H^1(\R^2)}+\|\p_1f_2\|_{H^1(\R^2)}).
\end{aligned}
\end{equation*}
Integrating the  inequality above in time, and then using H\"{o}lder's inequality  can yield the desired result.
\end{proof}
\vskip .2in
\section{Energy estimate I}
\label{s3}
\vskip .3in
In this section, we prove some a priori estimates, which are given by the following lemma:
\begin{lemma}\label{l3.1}
Let $(\omega,\theta)$ be sufficiently smooth functions which solves
(\ref{DBE}) and satisfy $(\omega_0,\theta_0)\in
\mathcal{H}^s(\R^2)\times H^s(\R^2)$, then there holds
\begin{equation}\label{E0}
E(T)+E_1(T)\le C\left(E(0)+E(T)^\frac{1}{2}E_1(T)+\left|\int_0^T I_1(t)dt\right|\right),
\end{equation}
where
\begin{equation*}
\begin{aligned}
E(T)\stackrel{\rm def}{=}&\|\omega\|_{L^\infty_T(\mathcal{H}^s)}^2
+\|\theta\|_{L^\infty_T(H^{s+1})}^2,\\
 E_1(T)\stackrel{\rm def}{=}& \|\omega\|_{L^2_T(\mathcal{H}^s)}^2+\|\p_1\theta\|_{L^2_T(H^s)}^2,\\
E(0)\stackrel{\rm def}{=}&\|\omega_0\|_{\mathcal{H}^s}^2+\|\theta_0\|_{H^{s+1}}^2,\\
I_1\stackrel{\rm def}{=}&\int \p_2 u_2\  (\p^{s+1}_2\theta)^2dx.
\end{aligned}
\end{equation*}
and  $C$ is a positive constant independent of $T$.
\end{lemma}
\begin{rem}\label{r3}
$\omega_0\in \dot{H}^{-1}$ is a natural condition, which is equal to $u_0\in L^2$. Thanks to this condition, together with $\theta_0\in L^2$, we can get the global  kinetic energy.
\end{rem}
\begin{proof}
Using the cancelation relations
$$(u\cdot\na \theta|\theta)=(u\cdot\na \omega|\omega)_{\dot{H}^{-1}}=0$$
and
$$(\p_1\theta|\omega)_{\dot{H}^s}
+(\p_1\Lambda^{-2}\omega|\theta)_{\dot{H}^{s+1}}=0,$$
we can get
\begin{equation}\label{3.1}
\begin{aligned}
\ddt(\|\omega\|_{\mathcal{H}^s}^2+\|\theta\|_{H^{s+1}}^2)+\|\omega\|_{\mathcal{H}^s}^2
=&-(u\cdot\nabla \omega|\omega)_{\mathcal{H}^s}-(u\cdot\nabla \theta|\theta)_{H^{s+1}}\\
=&-(u\cdot\nabla \omega|\omega)_{\dot{H}^s}\underbrace{-(u\cdot\nabla \theta|\theta)_{\dot{H}^{s+1}}}_{\stackrel{\rm def}{=}I}.
\end{aligned}
\end{equation}
Using $(u\cdot\na \Lambda^s\omega|\Lambda^s\omega)=0$ and
(\ref{kp}), we have
\begin{equation*}
\begin{aligned}
-(u\cdot\nabla \omega|\omega)_{\dot{H}^s}
=&\ -\int [\Lambda^s, u\cdot \nabla]  \omega \Lambda^s \omega dx\\
\le&\ \|[\Lambda^s, u\cdot \nabla]  \omega\|_{L^2}\|\Lambda^s \omega\|_{L^2}\\
\le&\ C\|\na u\|_{L^\infty}\|\Lambda^s\omega\|_{L^2}^2
+C\|\na \omega\|_{L^4}\|\Lambda^s u\|_{L^4} \|\Lambda^s\omega\|_{L^2}\\
\le&\  C\|\omega\|_{H^s}^3.
\end{aligned}
\end{equation*}
For the  estimate of $I$, using $(u\cdot\na \p^{s+1}\theta|\p^{s+1}\theta)=0$, we  obtain
$$I
=-\int \p^{s+1}(u\cdot\nabla \theta)\ \p^{s+1}\theta dx
=
-\sum_{1\le \alpha\le s+1}C_{s+1}^\alpha
\int \p^\alpha u\cdot\nabla \p^{s+1-\alpha}\theta\  \p^{s+1}\theta dx.$$
Depending on the derivatives $\p^{s+1}$, we will split the estimate into two cases: (1) $\p^{s+1}$ including at least one derivative on $x_1$  and (2) $\p^{s+1}=\p_2^{s+1}$.  For the case (1), it is easy to get
$$I\le C\|\omega\|_{\mathcal{H}^s}\|\p_1\theta\|_{H^s}\|\theta\|_{H^{s+1}}.$$
For the case (2), thanks to
$$
\sum_{1\le \alpha\le s+1}C_{s+1}^\alpha
\int \p_2^\alpha u_1 \p_1\p_2^{s+1-\alpha}\theta \p_2^{s+1}\theta dx
\le C\|\omega\|_{H^s}\|\p_1\theta\|_{H^s}\|\theta\|_{H^{s+1}}
$$
and
\begin{equation*}
\begin{aligned}
&\ \ \sum_{2\le \alpha\le s+1}C_{s+1}^\alpha
\int \p_2^\alpha u_2 \ \p_2^{s+2-\alpha}\theta\  \p_2^{s+1}\theta\ dx\\
=& -\sum_{2\le \alpha\le s+1}C_{s+1}^\alpha
\int \p_1\p_2^{\alpha-1} u_1\  \p_2^{s+2-\alpha}\theta \ \p_2^{s+1}\theta\ dx\\
=& \sum_{2\le \alpha\le s+1}C_{s+1}^\alpha(
\int\p_2^{\alpha-1} u_1\  \p_1\p_2^{s+2-\alpha}\theta \ \p_2^{s+1}\theta\ dx
+\int\p_2^{\alpha-1} u_1\ \p_2^{s+2-\alpha}\theta\  \p_1\p_2^{s+1}\theta\ dx)\\
=&\sum_{2\le \alpha\le s+1}C_{s+1}^\alpha(
\int\p_2^{\alpha-1} u_1\  \p_1\p_2^{s+2-\alpha}\theta \ \p_2^{s+1}\theta\ dx
-\int\p_2^\alpha u_1\  \p_2^{s+2-\alpha}\theta\  \p_1\p_2^s\theta\ dx\\
&-\int\p_2^{\alpha-1} u_1\  \p_2^{s+3-\alpha}\theta\  \p_1\p_2^s\theta\ dx)\\
\le & C\|\omega\|_{\mathcal{H}^s}\|\p_1\theta\|_{H^s}\|\theta\|_{H^{s+1}},
\end{aligned}
\end{equation*}
where we have used $\p_2u_2=-\p_1u_1$ and  integration by parts two times,
then
\begin{equation*}
\begin{aligned}
I=&\  - \sum_{1\le \alpha\le s+1} C_{s+1}^\alpha \left(\int \p_2^\alpha u_1\ \p_1\p_2^{s+1-\alpha}\theta\ \p_2^{s+1}\theta dx
+\int \p_2^\alpha u_2\ \p_2^{s+2-\alpha}\theta\ \p_2^{s+1}\theta dx\right)\\
\le&\  C\|\omega\|_{\mathcal{H}^s}\|\p_1\theta\|_{H^s}\|\theta\|_{H^{s+1}}
+I_1.
\end{aligned}
\end{equation*}
The estimate of $I_1$ will be given in  section \ref{s4}. So (\ref{3.1}) reduces to
\begin{equation}\label{3.0}
\frac{1}{2}\ddt(\|\omega\|_{\mathcal{H}^s}^2+\|\theta\|_{H^{s+1}}^2)
+\|\omega\|_{\mathcal{H}^s}^2
\le C (\|\omega\|_{\mathcal{H}^s}\|\p_1\theta\|_{H^s}\|\theta\|_{H^{s+1}}
+\|\omega\|_{\mathcal{H}^s}^3)+I_1.
\end{equation}
Next, we will find the dissipation of $\theta$. As a matter of fact, we have
$$\|\p_1\theta\|_{H^s}^2=(\p_t\omega|\p_1\theta)_{H^s}
+(u\cdot\nabla \omega|\p_1\theta)_{H^s}+(\omega|\p_1\theta)_{H^s}.
$$
Using (\ref{DBE})$_2$, we have
\begin{equation*}
\begin{aligned}
(\p_t\omega|\p_1\theta)_{H^s}
=&\ddt(\omega|\p_1\theta)_{H^s}-(\omega|\p_1\p_t\theta)_{H^s}\\
=&\ddt(\omega|\p_1\theta)_{H^s}
-(\omega|\p_1^2\Lambda^{-2}\omega)_{H^s}
+(\p_1(u\cdot\nabla \theta)|\omega)_{H^s}\\
=&\ddt(\omega|\p_1\theta)_{H^s}
+\|\p_1\Lambda^{-1}\omega\|_{H^s}^2
+(\p_1(u\cdot\nabla \theta)|\omega)_{H^s}.
\end{aligned}
\end{equation*}
So there exists a positive $C_0$ such that
\begin{equation}\label{3.3}
\begin{aligned}
\|\p_1\theta\|_{H^s}^2-\ddt(\omega|\p_1\theta)_{H^s}
=&(\omega|\p_1\theta)_{H^s}+\|\p_1\Lambda^{-1}\omega\|_{H^s}^2\\
&+(\p_1(u\cdot\nabla \theta)|\omega)_{H^s}+(u\cdot\nabla \omega|\p_1\theta)_{H^s}\\
\le& C_0\|\omega\|_{H^s}^2+\frac{1}{2}\|\p_1\theta\|_{H^s}^2
+(\p_1u\cdot\nabla \theta|\omega)_{H^s}\\
&+(u\cdot\nabla\p_1 \theta|\omega)_{\dot{H}^s}+(u\cdot\nabla \omega|\p_1\theta)_{\dot{H}^s},
\end{aligned}
\end{equation}
where we have used
$$(u\cdot\na \omega|\p_1\theta)+(u\cdot\na \p_1\theta|\omega)=0.$$
It is easy to get
$$
(\p_1u\cdot\nabla \theta|\omega)_{H^s}
\le C\|\omega\|_{H^s}^2\|\theta\|_{H^{s+1}}.
$$
Using the cancelation property
$$\int u\cdot\nabla \Lambda^s \omega\  \p_1\Lambda^s\theta dx
+\int u\cdot\nabla \p_1\Lambda^s\theta\ \Lambda^s\omega dx=0,
$$
then
\begin{equation*}
\begin{aligned}
&\ \ (u\cdot\nabla\p_1 \theta|\omega)_{\dot{H}^s}+(u\cdot\nabla \omega|\p_1\theta)_{\dot{H}^s}\\
=& \int [\Lambda^s,u\cdot\nabla]\p_1 \theta\Lambda^s\omega dx
+\int [\Lambda^s,u\cdot\nabla ]\omega \p_1\Lambda^s\theta dx\\
\le& \|[\Lambda^s,u\cdot\nabla]\p_1 \theta\|_{L^2}\|\omega\|_{H^s}
+\|[\Lambda^s,u\cdot\nabla ]\omega\|_{L^2} \|\p_1\theta\|_{H^s}.
\end{aligned}
\end{equation*}
Thanks to (\ref{kpv}) and interpolation inequalities, we have
\begin{equation*}
\begin{aligned}
\|[\Lambda^s,u\cdot\nabla]\p_1 \theta\|_{L^2}
\le& C(\|\nabla u\|_{L^\infty}\|\p_1 \theta\|_{H^s}
+\|\nabla \p_1\theta\|_{L^4}\|\Lambda^s u\|_{L^4})\\
\le& C(\|\nabla u\|_{L^\infty}\|\p_1 \theta\|_{H^s}
+\|\nabla \p_1\theta\|_{L^2}^\frac{1}{2}\|\Delta\p_1\theta\|_{L^2}^\frac{1}{2}
\|\Lambda^su\|_{L^2}^\frac{1}{2}\|\Lambda^s\omega\|_{L^2}^\frac{1}{2})\\
\le& C\|\omega\|_{H^s}\|\p_1 \theta\|_{H^s}.
\end{aligned}
\end{equation*}
Similarly, we can obtain
$$
\|[\Lambda^s,u\cdot\nabla ]\omega\|_{L^2}
\le C\|\omega\|_{H^s}^2.
$$
Thus,
$$(u\cdot\nabla\p_1 \theta|\omega)_{\dot{H}^s}+(u\cdot\nabla \omega|\p_1\theta)_{\dot{H}^s}
\le C\|\omega\|_{H^s}^2\|\p_1\theta\|_{H^s}.$$
Plugging  the estimates above
into (\ref{3.3}) yields
\begin{equation}\label{3.4}
\frac{1}{2}\|\p_1\theta\|_{H^s}^2
-\ddt(\omega|\p_1\theta)_{H^s}\le
C_0\|\omega\|_{H^s}^2+C\|\omega\|_{H^s}^2\|\theta\|_{H^{s+1}}.
\end{equation}
Multiplying (\ref{3.0}) by $2C_0$, and adding the resulting inequality to (\ref{3.4}), we can get
\begin{equation}\label{E1}
\begin{aligned}
&\ \ \ddt\{C_0(\|\omega\|_{\mathcal{H}^s}^2+\|\theta\|_{H^{s+1}}^2)
-(\omega|\p_1\theta)_{H^s}\}\\
&\ \ +C_0\|\omega\|_{\mathcal{H}^s}^2
+\frac{1}{2}\|\p_1\theta\|_{H^s}^2\\
\le&  C(\|\omega\|_{\mathcal{H}^s}+\|\theta\|_{H^{s+1}})
(\|\omega\|_{\mathcal{H}^s}^2+\|\p_1\theta\|_{H^s}^2)
+2C_0I_1.
\end{aligned}
\end{equation}
 Using
$$C_0(\|\omega\|_{\mathcal{H}^s}^2+\|\theta\|_{H^{s+1}}^2)
-(\omega|\p_1\theta)_{H^s}\thickapprox \|\omega\|_{\mathcal{H}^s}^2+\|\theta\|_{H^{s+1}}^2,$$
 and integrating (\ref{E1}) in time can lead to
the desired estimate (\ref{E0}).
\end{proof}
\vskip .2in
\section{The estimate of $\int_0^T\int \p_2u_2(\p_2^{s+1}\theta)^2 dxdt$}
\label{s4}
\vskip .3in
In this section, we bound $\int_0^T I_1(t)dt$,
\begin{lemma}\label{l4.1}
Under the conditions in Lemma \ref{l3.1}, then there holds
\begin{equation}\label{e2}
\int_0^T\int \p_2u_2(\p_2^{s+1}\theta)^2 dxdt\\
\le C(M(T)
+\mathfrak{K}(T)),
\end{equation}
where
\begin{equation}\label{K}
\mathfrak{K}(T)\stackrel{\rm def}{=}|\int_0^T\int u_2\p_2^2\theta  (\p_2^{s+1}\theta)^2dxdt|.
\end{equation}
 and
$$M(T)\stackrel{\rm def}{=}E(T)^\frac{1}{2}\left(E(T)+E_1(T)
+E(T)^\frac{3}{2}+E(T)E_1(T)\right).$$
Furthermore, we have
\begin{equation}\label{100}
E(T)+\frac{19}{20}E_1(T)
\le C(E(0)+M(T))+CE(T)^\frac{5}{3}
\int_0^T\|u_2\|_{L^\infty}^\frac{4}{3}dt.
\end{equation}
Here $C$ is a positive constant independent of $T$.
\end{lemma}
\begin{rem}\label{r2}
The quantity $\int_0^T \|\p_2u_2\|_{L^\infty} dt$ seems difficult to be bounded, but instead section \ref{s5} gives a bound for $\int_0^T \|\p_2u_2\|_{L^\infty}^\frac{4}{3} dt$. This motivates the setting and results of section \ref{s4}.
\end{rem}
\begin{proof}
We shall find a new way to bound this integral. In fact, using $-u_2=\p_t\theta+u\cdot\nabla \theta$, we have
\begin{equation*}
\begin{aligned}
&\ \ \ \int_0^T\int \p_2u_2(\p_2^{s+1}\theta)^2 dxdt\\
=&-\int_0^T\int \p_2\p_t\theta (\p_2^{s+1}\theta)^2  dxdt-
\int_0^T\int \p_2(u\cdot\nabla \theta) (\p_2^{s+1}\theta)^2dxdt\\
\stackrel{\rm def}{=}& J_1+J_2.
\end{aligned}
\end{equation*}
{\bf $\bullet$ The estimate of $J_1$}
\vskip .1in
 Using $\p_t\theta=-u\cdot\nabla \theta-u_2$, we can obtain
\begin{equation*}
\begin{aligned}
 J_1=&-\int \p_2\theta (\p_2^{s+1}\theta)^2  dx \mid_0^T+2
 \int_0^T\int \p_2\theta \p_2^{s+1}\p_t\theta\p_2^{s+1}\theta dx dt\\
 \le& \|\theta\|_{L^\infty_T(H^{s+1})}^3
 +2(-J_{11}+J_{12})\\
 \le& E(T)^\frac{3}{2}+2(-J_{11}+J_{12}),
\end{aligned}
\end{equation*}
where
$$J_{11}\stackrel{\rm def}{=} \int_0^T\int \p_2\theta \p_2^{s+1}u_2\p_2^{s+1}\theta dxdt,\ J_{12}\stackrel{\rm def}{=}
-\int_0^T\int \p_2\theta \p_2^{s+1}(u\cdot\nabla \theta)\p_2^{s+1}\theta dxdt.$$
For $J_{11}$, using $\p_2u_2=-\p_1u_1$ and integrating by parts two times lead to
\begin{equation*}
\begin{aligned}
J_{11}=&-\int_0^T\int \p_2\theta \p_2^s\p_1u_1\p_2^{s+1}\theta dxdt\\
=&\int_0^T\int \p_1\p_2\theta \p_2^su_1\p_2^{s+1}\theta dxdt
+\int_0^T\int\p_2\theta \p_2^su_1\p_1\p_2^{s+1}\theta dxdt\\
=&\int_0^T\int \p_1\p_2\theta \p_2^su_1\p_2^{s+1}\theta dxdt
-\int_0^T\int\p_2^2\theta\p_2^su_1\p_1\p_2^s\theta dxdt\\
&-\int_0^T\int\p_2\theta\p_2^{s+1}u_1\p_1\p_2^s\theta dxdt\\
\le& C\|\p_1\theta\|_{L^2_T(H^s)}\|\omega\|_{L^2_T(H^s)}
\|\theta\|_{L^\infty_T(H^{s+1})}\\
\le& CE(T)^\frac{1}{2}E_1(T).
\end{aligned}
\end{equation*}
For $J_{12}$, we have
\begin{equation*}
\begin{aligned}
J_{12}=&-\sum_{0\le \alpha\le s+1}C_{s+1}^\alpha\int_0^T\int \p_2\theta \p_2^\alpha u\cdot\na \p_2^{s+1-\alpha}\theta \p_2^{s+1}\theta dxdt\\
=&-\sum_{1\le \alpha\le s+1}C_{s+1}^\alpha\int_0^T\int \p_2\theta \p_2^\alpha u\cdot\na \p_2^{s+1-\alpha}\theta \p_2^{s+1}\theta dxdt\\
&-\frac{1}{2}\int_0^T\int \p_2\theta  u\cdot\na (\p_2^{s+1}\theta)^2dxdt\\
=&-\sum_{1\le \alpha\le s+1}C_{s+1}^\alpha\int_0^T\int \p_2\theta \p_2^\alpha u\cdot\na \p_2^{s+1-\alpha}\theta \p_2^{s+1}\theta dxdt\\
&+\frac{1}{2}\int_0^T\int u\cdot\na\p_2\theta  (\p_2^{s+1}\theta)^2dxdt\\
\stackrel{\rm def}{=}& J_{121}+J_{122}.
\end{aligned}
\end{equation*}
By using $\p_2u_2=-\p_1u_1$ and integration by parts, we have
\begin{equation*}
\begin{aligned}
J_{121}=&-\sum_{1\le \alpha\le s+1}C_{s+1}^\alpha\int_0^T\int \{\p_2\theta  \p_2^\alpha u_1\  \p_1 \p_2^{s+1-\alpha}\theta\  \p_2^{s+1}\theta\\
&+ \p_2\theta\  \p_2^\alpha u_2\  \p_2^{s+2-\alpha}\theta\  \p_2^{s+1}\theta\}dxdt\\
=&-\sum_{1\le \alpha\le s+1}C_{s+1}^\alpha\int_0^T\int \{\p_2\theta \ \p_2^\alpha u_1\ \p_1 \p_2^{s+1-\alpha}\theta\  \p_2^{s+1}\theta\\
&-\p_2\theta\  \p_2^{\alpha-1}\p_1 u_1\  \p_2^{s+2-\alpha}\theta\  \p_2^{s+1}\theta\}dxdt\\
=&-\sum_{1\le \alpha\le s+1}C_{s+1}^\alpha\int_0^T\int \p_2\theta \ \p_2^\alpha u_1\ \p_1 \p_2^{s+1-\alpha}\theta\  \p_2^{s+1}\theta dxdt\\
&+\sum_{2\le \alpha\le s+1}C_{s+1}^\alpha\int_0^T\int
\p_2\theta\  \p_2^{\alpha-1}\p_1 u_1\  \p_2^{s+2-\alpha}\theta\  \p_2^{s+1}\theta dxdt\\
&-(s+1)\int_0^T\int
\p_2\theta\  \p_2 u_2\  (\p_2^{s+1}\theta)^2 dxdt\\
\stackrel{\rm def}{=}& K_1+K_2+K_3.
\end{aligned}
\end{equation*}
By using the previous approach,
one can get the estimate as follows:
\begin{equation}\label{App1}
K_1\le CE(T)E_1(T),\ \ K_2\le CE(T)E_1(T),
\end{equation}
whose proof is given  in the appendix.
As for the estimate of $K_3$,  using the equation of $\theta$ two times and $\p_2u_2=-\p_1u_1$, we have
\begin{equation*}
\begin{aligned}
K_3\le&(s+1)\int_0^T\int \p_2\theta\p_2\p_t\theta (\p_2^{s+1}\theta)^2dxdt\\
&+(s+1)\int_0^T\int \p_2\theta \p_2(u\cdot\na \theta)(\p_2^{s+1}\theta)^2dxdt\\
=& \frac{s+1}{2}\int (\p_2\theta)^2(\p_2^{s+1}\theta)^2dx\mid_0^T\\
&-(s+1)\int_0^T\int(\p_2\theta)^2 \p_2^{s+1}\theta\p_2^{s+1}\p_t\theta dxdt\\
&+(s+1)\int_0^T\int \p_2\theta \p_2(u\cdot\na \theta)(\p_2^{s+1}\theta)^2dxdt\\
\le& CE(T)^2+(s+1)\int_0^T\int(\p_2\theta)^2 \p_2^{s+1}\theta\p_2^{s}\p_1u_1 dxdt\\
&+(s+1)\int_0^T\int(\p_2\theta)^2 \p_2^{s+1}\theta\p_2^{s+1}(u\cdot\na \theta) dxdt\\
&-(s+1)\int_0^T\int \p_2\theta \p_2(u\cdot\na \theta)(\p_2^{s+1}\theta)^2dxdt\\
\stackrel{\rm def}{=}& CE(T)^2+\sum_{i=1}^3K_{3i}.
\end{aligned}
\end{equation*}
Integrating by parts two times, one can get
$$K_{31}\le CE(T)E_1(T).$$
Integrating by parts and using  Lemma \ref{l2.1} with $(p,q)
=(2,\infty)$ and $(p,q)
=(\infty,2)$, we have
\begin{equation*}
\begin{aligned}
K_{32}\le& C\left|\int_0^T\int (\p_2\theta)^2u\cdot\nabla (\p_2^{s+1}\theta)^2dxdt\right|\\
&+C\sum_{1\le \alpha\le s+1}\left|\int_0^T\int (\p_2\theta)^2
\p^\alpha u\cdot \nabla \p^{s+1-\alpha}\theta \p_2^{s+1}\theta  dxdt \right|\\
\le& C\left|\int_0^T\int u\cdot\nabla(\p_2\theta)^2 (\p_2^{s+1}\theta)^2dxdt\right|\\
&+C\sum_{1\le \alpha\le s-1}\left|\int_0^T\int (\p_2\theta)^2
\p^\alpha u\cdot \nabla \p^{s+1-\alpha}\theta \p_2^{s+1}\theta  dxdt \right|\\
&+C\sum_{s\le \alpha\le s+1}\left|\int_0^T\int (\p_2\theta)^2
\p^\alpha u\cdot \nabla \p^{s+1-\alpha}\theta \p_2^{s+1}\theta  dxdt \right|\\
\le& C\|\p_1\theta\|_{L^2_T(H^s)}\|\theta\|_{L^\infty_T(H^{3})}
\|u\|_{L^2_T(H^{s+1})}\|\theta\|_{L^\infty_T(H^{s+1})}^2\\
\le& CE(T)^\frac{3}{2}E_1(T).
\end{aligned}
\end{equation*}
Similarly, using  Lemma \ref{l2.1} with $(p,q)
=(2,\infty)$, we can get
\begin{equation*}
\begin{aligned}
K_{33}\le& C\left|\int_0^T\int \p_2\theta \p_2u\cdot\na \theta (\p_2^{s+1}\theta)^2 dxdt\right|\\
&+C\left|\int_0^T\int \p_2\theta u\cdot\na \p_2\theta (\p_2^{s+1}\theta)^2 dxdt\right|\\
\le& C\|\p_1\theta\|_{L^2_T(H^s)}\|\theta\|_{L^\infty_T(H^{3})}
\|u\|_{L^2_T(H^3)}\|\theta\|_{L^\infty_T(H^{s+1})}^2\\
\le& CE(T)^\frac{3}{2}E_1(T).
\end{aligned}
\end{equation*}
So we can get the estimate of $K_3$,  and combining  with the estimates of $K_1$ and $K_2$ yields
$$J_{121}\le C\left(E(T)^2+E(T)E_1(T)+E(T)^\frac{3}{2}E_1(T)\right).$$
For the estimate of $J_{122}$, we have
\begin{equation*}
\begin{aligned}
J_{122}\le& C|\int_0^T\int u_1\p_1\p_2\theta  (\p_2^{s+1}\theta)^2dxdt|\\
&+|\int_0^T\int u_2\p_2^2\theta  (\p_2^{s+1}\theta)^2dxdt|\\
\le& C\|u\|_{L^2_T(L^\infty)}\|\p_1\theta\|_{L^2_T(H^3)}
\|\theta\|_{L^\infty_T(H^{s+1})}^2+C\mathfrak{K}(T)\\
\le& CE(T)E_1(T)+C\mathfrak{K}(T).
\end{aligned}
\end{equation*}
Hence,
\begin{equation}\label{J1}
J_1\le CE(T)^\frac{1}{2}\left(E(T)+E(T)^\frac{1}{2}E_1(T)+E(T)^\frac{3}{2}+E(T)E_1(T)\right)
+C\mathfrak{K}(T).
\end{equation}
\vskip .1in
{\bf $\bullet$ The estimate of $J_2$} We have
\begin{equation*}
\begin{aligned}
J_2\le& \left|\int_0^T\int \p_2u_1\p_1 \theta (\p_2^{s+1}\theta)^2dxdt\right|\\
&+\left|\int_0^T\int \p_2u_2\p_2 \theta (\p_2^{s+1}\theta)^2dxdt\right|.
\end{aligned}
\end{equation*}
It is easy to get
$$\left|\int_0^T\int \p_2u_1\p_1 \theta (\p_2^{s+1}\theta)^2dxdt\right|\le E(T)E_1(T),$$
while the second term can be bounded as the estimate of $K_3$. So we can get
$$J_2\le CE(T)\left(E_1(T)+E(T)^\frac{1}{2}E_1(T)+E(T)\right).$$
We can get the desired result (\ref{e2}) by combining with the estimates of $J_1$ and $J_2$.
\vskip.1in
Next, we show (\ref{100}).
One can deduce from Lemma \ref{l3.1} and Lemma \ref{l4.1} that
$$E(T)+E_1(T)\le C\left(E(0)+M(T)+\mathfrak{K}(T)\right).$$
It is easy to see that $M(T)$ is a good term, since $M(T)$ can be bounded by $E(T)^\alpha E_1(T)^\beta$ for $\alpha+\beta>1$ and $(\alpha,\beta)\in [0,\infty)^2$.  Next, it suffices to show the estimate of $\mathfrak{K}(T)$. Thanks to
$$\|\p_2^2\theta\|_{L^\infty}\le C\|\p_2^2\theta\|_{H^1}^\frac{1}{2}\|\p_1\p_2^2\theta\|_{H^1}^\frac{1}{2},$$
 we have
\begin{equation}\label{4.1}
\begin{aligned}
\mathfrak{K}(T)\le & \int_0^T\|u_2\|_{L^\infty}\|\p_2^2\theta\|_{L^\infty}
\|\theta\|_{H^{s+1}}^2dt\\
\le& C\int_0^T\|u_2\|_{L^\infty}
\|\p_2^2\theta\|_{H^1}^\frac{1}{2}\|\p_1\p_2^2\theta\|_{H^1}^\frac{1}{2}
\|\theta\|_{H^{s+1}}^2dt\\
\le& CE(T)^\frac{5}{4}\|\p_1\theta\|_{L^2_T(H^s)}^\frac{1}{2}
\|u_2\|_{L^\frac{4}{3}_T(L^\infty)}\\
\le& CE(T)^\frac{5}{3}
\|u_2\|_{L^\frac{4}{3}_T(L^\infty)}^\frac{4}{3}+\frac{1}{20}
E_1(T),
\end{aligned}
\end{equation}
which yields (\ref{100}).
\end{proof}
To close the estimate of (\ref{100}), we shall bound  $\int_0^T\|u_2\|_{L^\infty}^\frac{4}{3}dt$, which is the main goal in the following section.
\vskip .2in
\section{Spectral analysis and the estimate of $\int_0^T\|u_2\|_{L^\infty}^\frac{4}{3}dt$}
\label{s5}
\vskip .3in
In this section, we get the expression of solution and then show the estimate of
$\int_0^T\|u_2\|_{L^\infty}^\frac{4}{3}dt$.
\begin{lemma}\label{ll1}
Under the conditions in Lemma \ref{l3.1}, then there holds
\begin{equation}\label{ADDD2}
\int_0^T\|u_2\|_{L^\infty}^\frac{4}{3}dt
\le C\mathfrak{M}(T),
\end{equation}
where
$$
\begin{aligned}
\mathfrak{M}(T)
=&
CE(T)^\frac{1}{2}
+CE_1(T)+C(E_1(T)^\frac{1}{2}+E_2(T)^\frac{1}{2})E_1(T)^\frac{1}{2}\\
&+C\left(E_1(T)^\frac{1}{2}+E_2(T)^\frac{1}{2}\right)
E(T)^\frac{1}{4}E_1(T)^\frac{1}{4},
\end{aligned}
$$
\begin{equation}\label{901}
E_2(T)=\|\Lambda^{-2}\omega\|_{L^2_T(L^2)}^2.
\end{equation}
 In addition, we have
\begin{equation}\label{200}
E(T)+\frac{19}{20}E_1(T)
\le C(E(0)+M(T))+CE(T)^\frac{5}{3}
\mathfrak{M}(T).
\end{equation}
Here $C$ is a positive constant independent of $T$.
\end{lemma}
\begin{proof}
The proof consists of the following two subsections.
We need first a diagonalization process.
\subsection{Spectral analysis}
To obtain  the estimate of $\int_0^T\|u_2\|_{L^\infty}^\frac{4}{3}dt$, we shall  first investigate the spectrum properties to the following system:
\begin{equation} \label{LBE}
\left\{
\begin{array}{l}
\partial_t \omega + \omega   = \p_1\theta+G,  \\
\partial_t \theta   =\p_1\Lambda^{-2}\omega+H,\\
G=-u\cdot\nabla \omega,\ H=-u\cdot\nabla \theta.
\end{array}
\right.
\end{equation}
Denote
$$A\stackrel{\rm def}{=}\left(
             \begin{array}{cc}
               -1 & -i\xi_1 \\
               -\frac{i\xi_1}{|\xi|^2} & 0 \\
             \end{array}
           \right),$$
then we can get from (\ref{LBE}) that
\begin{equation}\label{5.1}
\p_t\left(
        \begin{array}{c}
          \widehat{\omega} \\
          \widehat{\theta} \\
        \end{array}
      \right)(\xi)
=A \left(
        \begin{array}{c}
          \widehat{\omega} \\
          \widehat{\theta} \\
        \end{array}
      \right)(\xi)
      +\left(
        \begin{array}{c}
          \widehat{G} \\
          \widehat{H} \\
        \end{array}
      \right)(\xi).
\end{equation}
One  can get the eigenvalues of the matrix $A$ as follows:
 \begin{equation}\label{eige}
\lambda_\pm=\left\{
\begin{array}{l}
 \frac{-1\pm |\xi|^{-1}\sqrt{|\xi|^2-4\xi_1^2}}{2},\ \ {\rm when} \ \  |\xi|\ge 2|\xi_1|,  \\
\frac{-1\pm i|\xi|^{-1}\sqrt{4\xi_1^2-|\xi|^2}}{2},\ \ {\rm when} \ \  |\xi|<2|\xi_1|
\end{array}
\right.
\end{equation}
and
\begin{equation}\label{5.2}
\mathbb{P}^{-1}A\mathbb{P}=\left(
                               \begin{array}{cc}
                                 \lambda_+ & 0 \\
                                 0 &  \lambda_-\\
                               \end{array}
                             \right),
\end{equation}
where the matrixes  $\mathbb{P}$ and $\mathbb{P}^{-1}$ are given by
$$\mathbb{P}=\left(
               \begin{array}{cc}
                 \lambda_+|\xi|^2 & \lambda_-|\xi|^2 \\
                 -i\xi_1 & -i\xi_1 \\
               \end{array}
             \right),
             \ \
            \mathbb{P}^{-1}=\frac{1}{|A|} \left(
               \begin{array}{cc}
                 -i\xi_1 & -\lambda_-|\xi|^2 \\
                 i\xi_1 & \lambda_+|\xi|^2 \\
               \end{array}
             \right).
$$
Thanks to (\ref{5.1}) and (\ref{5.2}), denote
$$W\stackrel{\rm def}{=}-i\xi_1 \widehat{\omega}-\lambda_-|\xi|^2
\widehat{\theta},\ \  V\stackrel{\rm def}{=}i\xi_1\widehat{\omega}
+\lambda_+|\xi|^2\widehat{\theta},$$
we have
\begin{equation*}
\left\{
\begin{array}{l}
\partial_t W   = \lambda_+ W+(-i\xi_1\widehat{G}-\lambda_-|\xi|^2\widehat{H}),  \\
\partial_t V   =\lambda_- V+i\xi_1\widehat{G}+\lambda_+|\xi|^2\widehat{H},\\
\end{array}
\right.
\end{equation*}
which is equal to
$$W=e^{\lambda_+t}W_0+\int_0^t e^{\lambda_+(t-\tau)}
(-i\xi_1\widehat{G}-\lambda_-|\xi|^2\widehat{H})d\tau,$$
$$
V=e^{\lambda_-t}V_0+\int_0^t e^{\lambda_-(t-\tau)}(i\xi_1\widehat{G}+\lambda_+|\xi|^2\widehat{H} )d\tau.
$$
So we can get
\begin{equation}\label{integral}
\begin{aligned}
\widehat{\omega}(\xi,t)=&M_1(t)\widehat{\omega}_0(\xi)
+M_2(t)\widehat{\theta}_0(\xi)\\
&+\int_0^t M_1(t-\tau)\widehat{G}(\xi)d\tau+
\int_0^t M_2(t-\tau)\widehat{H}(\xi)d\tau,
\end{aligned}
\end{equation}
where
\begin{equation}\label{M1}
M_1(t)\stackrel{\rm def}{=}\left\{
\begin{array}{l}
-\frac{|\xi|
(\lambda_-e^{\lambda_-t}-\lambda_+
e^{\lambda_+t})}{\sqrt{|\xi|^2-4\xi_1^2}},\ \ {\rm when} \ \ |\xi|\ge 2|\xi_1|,  \\
-\frac{|\xi|
(\lambda_-e^{\lambda_-t}-\lambda_+
e^{\lambda_+t})}{i\sqrt{4\xi_1^2-|\xi|^2}},\ \ {\rm when} \ \ |\xi|< 2|\xi_1|,
\end{array}
\right.
\end{equation}
and
\begin{equation}\label{M2}
M_2(t)\stackrel{\rm def}{=}\left\{
\begin{array}{l}
\frac{i\xi_1|\xi|
(e^{\lambda_-t}-
e^{\lambda_+t})}{\sqrt{|\xi|^2-4\xi_1^2}},\ \ {\rm when} \ \ |\xi|\ge 2|\xi_1|,  \\
\frac{\xi_1|\xi|
(e^{\lambda_-t}-
e^{\lambda_+t})}{\sqrt{4\xi_1^2-|\xi|^2}},\ \ {\rm when} \ \ |\xi|< 2|\xi_1|.
\end{array}
\right.
\end{equation}
\subsection{The estimate of $\int_0^T\|u_2\|_{L^\infty}^\frac{4}{3}dt$}
Using $\|f\|_{L^\infty}\le \|\widehat{f}(\xi)\|_{L^1}$, we have
\begin{equation}\label{5.3}
\begin{aligned}
\int_0^T\|u_2\|_{L^\infty}^\frac{4}{3}dt
\le& \int_0^T\|\widehat{u_2}(\xi)\|_{L^1}^\frac{4}{3}dt\\
\le& \int_0^T\|\widehat{u_2}(\xi)\|_{L^1(D_1)}^\frac{4}{3}dt\\
&+\int_0^T\|\widehat{u_2}(\xi)\|_{L^1(D_2)}^\frac{4}{3}dt\\
&+\int_0^T\|\widehat{u_2}(\xi)\|_{L^1(D_3)}^\frac{4}{3}dt,
\end{aligned}
\end{equation}
where
\begin{equation*}
\begin{aligned}
D_1\stackrel{\rm def}{=}&\{\xi\in \R^2:\ |\xi|\ge 3|\xi_1|\},\\
D_2\stackrel{\rm def}{=}&\{\xi\in \R^2:\ 2|\xi_1|\le |\xi|< 3|\xi_1|\},\\
D_3\stackrel{\rm def}{=}&\{\xi\in \R^2:\ |\xi|< 2|\xi_1|\}.
\end{aligned}
\end{equation*}
{\bf Case 1 $\xi \in D_1$}
\vskip .1in
  In this case, one can get from (\ref{eige}), (\ref{M1}) and (\ref{M2}) that
  \begin{equation}\label{P1}
  \left\{
\begin{array}{l}
\lambda_+\le -\frac{\xi_1^2}{|\xi|^2},\ \lambda_-\le -\frac{1}{2},  \\
M_1(t)\le C(e^{-\frac{1}{2}t}+\frac{\xi_1^2}{|\xi|^2}
e^{-\frac{\xi_1^2}{|\xi|^2}t}),\\
M_2(t)\le C|\xi_1|(e^{-t}+e^{-\frac{\xi_1^2}{|\xi|^2}t}).
\end{array}
\right.
\end{equation}
Consequently, using (\ref{integral}), we can obtain
\begin{equation*}
\begin{aligned}
\left\|\|\frac{\xi_1}{|\xi|^2}\widehat{\omega}\|_{L^1(D_1)}
\right\|_{L^\frac{4}{3}_T} \le & \left\|\|\frac{\xi_1}{|\xi|^2}M_1(t)\widehat{\omega}_0\|_{L^1(D_1)}\right\|_{L^\frac{4}{3}_T}
+\left\|\|\frac{\xi_1}{|\xi|^2}M_2(t)
\widehat{\theta}_0\|_{L^1(D_1)}\right\|_{L^\frac{4}{3}_T}\\
&+\left\|\int_0^t\|\frac{\xi_1}{|\xi|^2}
M_1(t-\tau)\widehat{G}\|_{L^1(D_1)}d\tau\right\|_{L^\frac{4}{3}_T}\\
&+\left\|\int_0^t\|\frac{\xi_1}
{|\xi|^2}M_2(t-\tau)\widehat{H}\|_{L^1(D_1)}d\tau\right\|_{L^\frac{4}{3}_T}\\
\stackrel{\rm def}{=}&L_1+L_2+L_3+L_4.
\end{aligned}
\end{equation*}
Using (\ref{P1}), we have
$$\|\frac{\xi_1}{|\xi|}M_1(t)\|_{L^\infty_\xi}\le
C\langle t\rangle^{-\frac{3}{2}},\
\|\frac{\xi_1}{|\xi|^2}M_2(t)\|_{L^\infty_\xi}\le
C\langle t\rangle^{-1},
$$
which yields
$$L_1\le C\left\|\langle t\rangle^{-\frac{3}{2}}\|\widehat{\Lambda^{-1}
\omega}\|_{L^1}\right\|_{L^\frac{4}{3}_T}
\le C\left\|\langle t\rangle^{-\frac{3}{2}}\|u\|_{H^{1+\eta}}\right\|_{L^\frac{4}{3}_T}
\le C\|u\|_{L^\infty_T(H^{1+\eta})}\le CE(T)^\frac{1}{2},$$
$$
L_2\le C\left\|\langle t\rangle^{-1}\|\widehat{\theta}\|_{L^1}\right\|_{L^\frac{4}{3}_T}
\le C\left\|\langle t\rangle^{-1}\|\theta\|_{H^{1+\eta}}\right\|_{L^\frac{4}{3}_T}
\le C\|\theta\|_{L^\infty_T(H^{1+\eta})}\le CE(T)^\frac{1}{2},
$$
where $0<\eta\ll1$,  and
$$L_3\le C\left\|\int_0^t \langle t-\tau\rangle^{-\frac{3}{2}}
\|u\otimes \omega\|_{H^{1+\eta}}d\tau\right\|_{L^\frac{4}{3}_T}.
$$
Let
$$f_1(s)=\langle s\rangle^{-\frac{3}{2}}\chi_{(0,\infty)}(s)\in L^p \ (p\ge 1),\
g_1(s)=\|u\otimes \omega\|_{H^{1+\eta}}\chi_{(0,T)},$$
by Young's inequality, (\ref{kp}) and $H^{1+\eta}(\R^2)\hookrightarrow L^\infty(\R^2)$, then  we have
\begin{equation*}
\begin{aligned}
L_3
\le &\|f_1\star g_1\|_{L^\frac{4}{3}(\R)}\\
\le& C\|f_1\|_{L^\frac{4}{3}(\R)}\|g\|_{L^1(\R)}\\
\le& C\left\|\|u\|_{L^\infty}\|\omega\|_{H^{1+\eta}}   +\|\omega\|_{L^\infty}\|u\|_{H^{1+\eta}}\right\|_{L^1_T}\\
\le& C\|u\|_{L^2_T(H^{1+\eta})}\|\omega\|_{L^2_T(H^{1+\eta})}\\
\le& CE_1(T).
\end{aligned}
\end{equation*}
For the last term $L_4$,  we shall first give some analysis on $\widehat{u\cdot\na \theta}$. Applying $u\cdot\nabla =u_1\p_1+u_2\p_2$, we can get
$$L_4\le C\left\|\int_0^t\langle t-\tau\rangle^{-1} \|\widehat{u_1\p_1\theta}\|_{L^1(D_1)}d\tau\right\|_{L^\frac{4}{3}_T}
+C\left\|\int_0^t  \|\frac{\xi_1}{|\xi|^2}M_2(t-\tau)\widehat{u_2\p_2\theta}\|_{L^1(D_1)}d\tau\right\|_{L^\frac{4}{3}_T}
\stackrel{\rm def}{=}L_{41}+L_{42},
$$
where
$L_{41}$ can be bounded like $L_3$, indeed, denote
$$f_2(s)=\langle s\rangle^{-1}\chi_{(0,\infty)}(s)\in L^p \ (p> 1),\
g_2(s)=\|u_1\p_1 \theta\|_{H^{1+\eta}}\chi_{(0,T)},$$
similarly, we also have
\begin{equation*}
\begin{aligned}
L_{41}\le& C\|f_2\star g_2\|_{L^\frac{4}{3}(\R)}\\
\le&
C\|f_2\|_{L^\frac{4}{3}(\R)}\|g_2\|_{L^1(\R)}\\
\le& C\|u_1\|_{L^2_T(H^{1+\eta})}\|\p_1\theta\|_{L^2_T(H^{1+\eta})}\\
\le& CE_1(T).
\end{aligned}
\end{equation*}
However, this strategy can not be used to   the estimate of $L_{42}$, due to the fact that
$$\|\p_2\theta\|_{L^2_T(H^{1+\eta})} $$
 can not be bounded by $E_1(T)$. In order to avoid this bad term,  we shall seek a different approach. We begin with the analysis of $\widehat{u_2\p_2\theta}$. Using $u_2=-\p_1\Lambda^{-2}\omega$, one has
 \begin{equation}\label{tech1}
 \begin{aligned}
 \widehat{u_2\p_2\theta}=&\int (\xi_1-\eta_1)|\xi-\eta|^{-2}
 \widehat{\omega}(\xi-\eta)\eta_2\widehat{\theta}(\eta)d\eta\\
 =&\xi_1\int |\xi-\eta|^{-2}
 \widehat{\omega}(\xi-\eta)\eta_2\widehat{\theta}(\eta)d\eta\\
 &-\int |\xi-\eta|^{-2}
 \widehat{\omega}(\xi-\eta)\eta_1\eta_2\widehat{\theta}(\eta)d\eta\\
 =&i\xi_1\int \widehat{\Lambda^{-2}\omega}(\xi-\eta)\widehat{\p_2\theta}(\eta)d\eta\\
 &+\int \widehat{\Lambda^{-2}\omega}(\xi-\eta)
 \widehat{\p_1\p_2\theta}(\eta)d\eta\\
 =&i\xi_1\mathcal{F}[\Lambda^{-2}\omega\p_2\theta]
 (\xi)+\mathcal{F}[\Lambda^{-2}\omega\p_1\p_2\theta](\xi),
 \end{aligned}
 \end{equation}
 with (\ref{M2}) leads to
\begin{equation*}
\begin{aligned}
\|\frac{\xi_1}{|\xi|^2}M_2(t-\tau)\widehat{u_2\p_2\theta}\|_{L^1(D_1)}
\le& C\langle t-\tau\rangle^{-\frac{3}{2}}\|\mathcal{F}
[\Lambda(\Lambda^{-2}\omega\p_2\theta)]\|_{L^1}\\
&+C\langle t-\tau\rangle^{-1}\|\mathcal{F}
[\Lambda^{-2}\omega\p_1\p_2\theta]\|_{L^1}
\\
\le& C\langle t-\tau\rangle^{-\frac{3}{2}}
\|\Lambda^{-2}\omega\ \p_2\theta\|_{H^{2+\eta}}\\
&+C\langle t-\tau\rangle^{-1}\|\Lambda^{-2}\omega\p_1\p_2\theta\|_{H^{1+\eta}}\\
\le& C\langle t-\tau\rangle^{-\frac{3}{2}}
\|\Lambda^{-2}\omega\ \p_2\theta\|_{H^{2+\eta}}\\
&+C\langle t-\tau\rangle^{-1}\|\Lambda^{-2}\omega\|_{H^{1+\eta}}
\|\p_1\p_2\theta\|_{H^{1+\eta}}\\
\stackrel{\rm def}{=}& N_1(t)+N_2(t).
  \end{aligned}
 \end{equation*}
 So we have
 \begin{equation}\label{5.4}
 L_{42}\le \|N_1(t)\|_{L^\frac{4}{3}_T}
 +\|N_2(t)\|_{L^\frac{4}{3}_T}.
 \end{equation}
 Denote $g_3(s)=\|\Lambda^{-2}\omega\|_{H^{1+\eta}}
\|\p_1\p_2\theta\|_{H^{1+\eta}}\chi_{(0,T)}$,
  by Young's inequality and (\ref{kp}), we have
\begin{equation*}
\begin{aligned}
\|N_2(t)\|_{L^\frac{4}{3}_T}
\le& C \|f_2\star g_3\|_{L^\frac{4}{3}}\\
\le& C \|f_2\|_{L^\frac{4}{3}(\R)}\|g_3\|_{L^1(\R)}\\
\le& C\|\Lambda^{-2}\omega\|_{L^2_T(H^{1+\eta})}
\|\p_1\p_2\theta\|_{L^2_T(H^{1+\eta})}\\
\le& C\|\Lambda^{-2}\omega\|_{L^2_T(H^{2+\eta})}E_1(T)^\frac{1}{2}\\
\le& C(E_2(T)^\frac{1}{2}+E_1(T)^\frac{1}{2})E_1(T)^\frac{1}{2}.
\end{aligned}
\end{equation*}
 \begin{rem}\label{ADD1}
 $E_2(T)$ plays an essential role in the estimate of $L_{42}$.  The regularity index -2 in (\ref{901}) may be slightly improved to $-1-\delta$ for some $\delta\in (0,1)$, but  it seems difficult to obtain the case $\delta=0$.
 \end{rem}
Denote
$g_4(s)=\|\Lambda^{-2}\omega\p_2\theta\|_{H^{2+\eta}}\chi_{(0,T)}$,
by Young's inequality,  (\ref{kp})
 and
 $$\|h\|_{L^\infty}\le C\|\|h\|_{L^2_{x_1}}^\frac{1}{2}
 \|\p_1h\|_{L^2_{x_1}}^\frac{1}{2}\|_{L^\infty_{x_2}}\le
 C\|h\|_{H^1}^\frac{1}{2}\|\p_1h\|_{H^1}^\frac{1}{2}, $$
 we can get
\begin{equation*}
\begin{aligned}
\|N_1(t)\|_{L^\frac{4}{3}_T}
\le& \|f_1\star g_4\|_{L^\frac{4}{3}}\\
\le&
C\|f_1\|_{L^1}\left\|\|\Lambda^{-2}\omega\  \p_2\theta\|_{H^{2+\eta}}\right\|_{L^\frac{4}{3}_T}\\
\le&C\|f_1\|_{L^1}\left\|\|\Lambda^{-2}\omega\  \p_2\theta\|_{H^3}\right\|_{L^\frac{4}{3}_T}\\
\le& C \left\| \|\Lambda^{-2}\omega\|_{L^2}\|\p_2\theta\|_{W^{3,\infty}}    \right\|_{L^\frac{4}{3}_T}\\
&+\left\| \|\Lambda^{-2} \omega\|_{H^{3}}\|\p_2\theta\|_{L^\infty} \right\|_{L^\frac{4}{3}_T}\\
\le& C\left\| \|\Lambda^{-2}\omega\|_{L^2}\|\p_2\theta\|_{H^{3}}^\frac{1}{2}
 \|\p_1\p_2\theta\|_{H^{3}}^\frac{1}{2}   \right\|_{L^\frac{4}{3}_T}\\
 &+\left\| \|\Lambda^{-2} \omega\|_{H^{3}}\|\p_2\theta\|_{H^1}^\frac{1}{2} \|\p_1\p_2\theta\|_{H^1}^\frac{1}{2}\right\|_{L^\frac{4}{3}_T}\\
 \le& C\|\Lambda^{-2}\omega\|_{L^2_T(L^2)}
 \|\p_2\theta\|_{L^\infty_T(H^{3})}^\frac{1}{2}
 \|\p_1\p_2\theta\|_{L^2_T(H^{3})}^\frac{1}{2}\\
 &+C\|\Lambda^{-2} \omega\|_{L^2_T(H^{3})}
 \|\p_2\theta\|_{L^\infty_T(H^1)}^\frac{1}{2}
  \|\p_1\p_2\theta\|_{L^2_T(H^1)}^\frac{1}{2}\\
\le& C\left(E_1(T)^\frac{1}{2}+E_2(T)^\frac{1}{2}\right)
E(T)^\frac{1}{4}E_1(T)^\frac{1}{4}.
\end{aligned}
\end{equation*}
Thus we have
\begin{equation*}
\begin{aligned}
L_{42}\le& C(E_1(T)^\frac{1}{2}+E_2(T)^\frac{1}{2})E_1(T)^\frac{1}{2}\\
&+C\left(E_1(T)^\frac{1}{2}+E_2(T)^\frac{1}{2}\right)
E(T)^\frac{1}{4}E_1(T)^\frac{1}{4},
\end{aligned}
\end{equation*}
with the estimate of $L_{41}$ leads to
\begin{equation*}
\begin{aligned}
L_4\le& C(E_1(T)^\frac{1}{2}+E_2(T)^\frac{1}{2})E_1(T)^\frac{1}{2}\\
&+C\left(E_1(T)^\frac{1}{2}+E_2(T)^\frac{1}{2}\right)
E(T)^\frac{1}{4}E_1(T)^\frac{1}{4}+CE_1(T).
\end{aligned}
\end{equation*}
Therefore, we can get
\begin{equation}\label{Domain1}
\int_0^T\|\widehat{u_2}(\xi)\|_{L^1(D_1)}^\frac{4}{3}dt
\le\mathfrak{M}(T).
\end{equation}
\vskip .1in
{\bf Case 2 $\xi\in D_2$}
\vskip .1in
In this case, we have
$$\frac{1}{9}<\frac{\xi_1^2}{|\xi|^2}\le \frac{1}{4},$$
and then
\begin{equation*}
\begin{aligned}
|M_1(t)|\le& \left|\frac{(\lambda_--\lambda_+)e^{\lambda_-t}}
{|\xi|^{-1}\sqrt{|\xi|^2-4\xi_1^2}}\right|
+\left|\frac{\lambda_+(e^{\lambda_+t}-e^{\lambda_-t})}
{|\xi|^{-1}\sqrt{|\xi|^2-4\xi_1^2}} \right|\\
\le& Ce^{-\frac{1}{2}t}+C\frac{|\lambda_+||\lambda_+-\lambda_-|t
e^{\lambda_+t}}{|\xi|^{-1}\sqrt{|\xi|^2-4\xi_1^2}}\\
\le&  Ce^{-\frac{1}{2}t}+C\frac{\xi_1^2}{|\xi|^2}te^{e^{-\frac{\xi_1^2}{|\xi|^2}
t  }}\\
\le& Ce^{-\frac{1}{18}t},
\end{aligned}
\end{equation*}
and
\begin{equation*}
\begin{aligned}
|M_2(t)|\le& C\frac{|\xi_1||\lambda_+-\lambda_-|te^{\lambda_+t}}
{|\xi|^{-1}\sqrt{|\xi|^2-4\xi_1^2}}\\
\le& C|\xi_1|te^{-\frac{\xi_1^2}{|\xi|^2}t}\\
\le& C|\xi_1|e^{-\frac{1}{18}t}.
\end{aligned}
\end{equation*}
It means $M_i(t)$ $(i=1,2)$ admits a faster decay than the {\bf Case 1}. So one can obtain by following the previous procedure that
$$\int_0^T\|\widehat{u_2}(\xi)\|_{L^1(D_2)}^\frac{4}{3}dt\le \mathfrak{M}(T).$$
\vskip .1in
{\bf Case 3 $\xi\in D_3$}
\vskip .1in
We have
\begin{equation*}
\begin{aligned}
|M_1(t)|\le& \left|\frac{|\xi|(\lambda_--\lambda_+)e^{\lambda_-t}}
{\sqrt{4\xi_1^2-|\xi|^2}}   \right|+\left|\frac{|\xi|\lambda_+(e^{\lambda_+t}-e^{\lambda_-t})}
{\sqrt{4\xi_1^2-|\xi|^2}} \right|\\
\le& C|e^{\lambda_-t}|
+C\left|\frac{e^{-\frac{1}{2}t}
(e^{i\frac{|\xi|^{-1}\sqrt{4\xi_1^2-|\xi|^2}}{2}t}-
e^{-i\frac{|\xi|^{-1}\sqrt{4\xi_1^2-|\xi|^2}}{2}t})}
{|\xi|^{-1}\sqrt{4\xi_1^2-|\xi|^2}}   \right|\\
\le& Ce^{-\frac{1}{2}t}+Ce^{-\frac{1}{2}t}\left|
\frac{\sin (\frac{|\xi|^{-1}\sqrt{4\xi_1^2-|\xi|^2}}{2}t)}
{|\xi|^{-1}\sqrt{4\xi_1^2-|\xi|^2}} \right|\\
\le & Ce^{-\frac{1}{2}t}+Ce^{-\frac{1}{2}t}t\\
\le& Ce^{-\frac{1}{4}t},
\end{aligned}
\end{equation*}
where we have used
$$\left|\frac{\sin x}{x}\right|\le 1.$$
Similarly, we also have
\begin{equation*}
\begin{aligned}
|M_2(t)|\le& C|\xi_1|e^{-\frac{1}{2}t}
\left| \frac{e^{i\frac{|\xi|^{-1}\sqrt{4\xi_1^2-|\xi|^2}}{2}t}
-e^{-i\frac{|\xi|^{-1}\sqrt{4\xi_1^2-|\xi|^2}}{2}t}}
{|\xi|^{-1}\sqrt{4\xi_1^2-|\xi|^2}} \right|\\
\le& C|\xi_1|e^{-\frac{1}{2}t}t\\
\le& C|\xi_1|e^{-\frac{1}{4}t}.
\end{aligned}
\end{equation*}
Then one can get by following the previous procedure line by line that
$$\int_0^T\|\widehat{u_2}(\xi)\|_{L^1(D_3)}^\frac{4}{3}dt\le \mathfrak{M}(T).$$
So
$$\int_0^T\|\widehat{u_2}(\xi)\|_{L^1}^\frac{4}{3}dt\le \mathfrak{M}(T).$$
Inserting the estimate  above in (\ref{100}), we can get (\ref{200}).
\end{proof}
\vskip .2in
\section{Energy estimate II}
\label{s6}
\vskip .3in
Due to the appearance of  $E_2(T)$, the  estimate (\ref{200}) is not closed. The following lemma will help us close this estimate.
\begin{lemma}\label{ll100}
Let $(\omega,\theta)$ be sufficiently smooth functions which solves
(\ref{DBE}) and satisfy $(\omega_0,\theta_0)\in
\mathbb{H}^s(\R^2)\times \mathcal{H}^{s+1}(\R^2)$, then there holds
\begin{equation}\label{E0000}
\begin{aligned}
&\ \ \|\omega\|_{L^\infty_T(\dot{H}^{-2})}^2
+\|\theta\|_{L^\infty_T(\dot{H}^{-1})}^2
+\int_0^T(\|\omega\|_{\dot{H}^{-2}}^2
+\|\p_1\theta\|_{\dot{H}^{-2}}^2)dt\\
\le& C(\|\omega_0\|_{\dot{H}^{-2}}^2+\|\theta_0\|_{\dot{H}^{-1}}^2)
+(A(T)^\frac{1}{2}+A(T))A_1(T),
\end{aligned}
\end{equation}
where $C$ is a positive constant independent of $T$.
\end{lemma}
\begin{proof}
Using energy method, we have
$$\frac{1}{2}\ddt\|\omega\|_{\dot{H}^{-2}}^2
+\|\omega\|_{\dot{H}^{-2}}^2=-(u\cdot\na \omega|\omega)_{\dot{H}^{-2}}+(\p_1\theta|\omega)_{\dot{H}^{-2}}$$
and
$$\frac{1}{2}\ddt\|\theta\|_{\dot{H}^{-1}}^2=-(u\cdot\na \theta|\theta)_{\dot{H}^{-1}}+(\p_1\Lambda^{-2}\omega|\theta)_{\dot{H}^{-1}}.$$
Thanks to  the cancelation property
$$(\p_1\theta|\omega)_{\dot{H}^{-2}}
+(\p_1\Lambda^{-2}\omega|\theta)_{\dot{H}^{-1}}=0,
$$
and
\begin{equation*}
\begin{aligned}
\Big|(u\cdot\na \omega|\omega)_{\dot{H}^{-2}}\Big|\le& C
\|\na\times(u\cdot\na u)\|_{\dot{H}^{-2}}\|\omega\|_{\dot{H}^{-2}}
\le C\|u\otimes u\|_{L^2}\|\omega\|_{\dot{H}^{-2}}\\
\le& C\|u\|_{L^4}^2\|\omega\|_{\dot{H}^{-2}}
\le C\|u\|_{L^2}\|\omega\|_{L^2}\|\omega\|_{\dot{H}^{-2}},
\end{aligned}
\end{equation*}
we  suffice to   show the estimate of $(u\cdot\na \theta|\theta)_{\dot{H}^{-1}}$. Denote $-\Delta f=\theta$,
using integration by parts many times,
 we have
\begin{equation*}
\begin{aligned}
(u\cdot\na \theta|\theta)_{\dot{H}^{-1}}
=&\int u\cdot\na \theta (-\Delta)^{-1}\theta
=-\int u\cdot\na \Delta f f\\
=&-\int \p_i (u\cdot\na\p_if) f+\int \p_iu\cdot\na \p_if f\\
=& \int u\cdot\na\p_if \p_i f+\int \p_iu\cdot\na \p_if f\\
=&\int \p_iu\cdot\na \p_if f
= \int \p_iu\cdot\na f \p_if \\
=& \int \p_iu_1 \p_1f \p_if+\int\p_1u_2 \p_2f \p_1f
+\int\p_2u_2 \p_2f \p_2f\\
=& \int \p_iu_1 \p_1f \p_if+\int\p_1u_2 \p_2f \p_1f
-\int\p_1u_1 \p_2f \p_2f\\
=&\int \p_iu_1 \p_1f \p_if+\int\p_1u_2 \p_2f \p_1f
+2\int u_1 \p_1\p_2f \p_2f,
\end{aligned}
\end{equation*}
then by interpolation inequality,
\begin{equation*}
\begin{aligned}
|(u\cdot\na \theta|\theta)_{\dot{H}^{-1}}|
\le&\  C\|\omega\|_{L^4}\|\p_1f\|_{L^2}\|\na f\|_{L^4}
+C\|u_1\|_{L^\infty}\|\p_1\p_2f\|_{L^2}\|\p_2f\|_{L^2}\\
=&\  C\|\omega\|_{L^4}\|\p_1\Lambda^{-2}\theta\|_{L^2}\|\na \Lambda^{-2}\theta\|_{L^4}
+C\|u_1\|_{L^\infty}\|\p_1\p_2\Lambda^{-2}\theta\|_{L^2}
\|\p_2\Lambda^{-2}\theta\|_{L^2}\\
\le&\  C\|u\|_{H^2}\|\p_1\Lambda^{-2}\theta\|_{H^1}
\|\Lambda^{-1}\theta\|_{H^2}.
\end{aligned}
\end{equation*}
By interpolation inequality, thus we can get
\begin{equation}\label{end1}
\begin{aligned}
\frac{1}{2}\ddt(\|\omega\|_{\dot{H}^{-2}}^2+\|\theta\|_{\dot{H}^{-1}}^2)
+\|\omega\|_{\dot{H}^{-2}}^2
\le& C\|u\|_{L^2}\|\omega\|_{L^2}\|\omega\|_{\dot{H}^{-2}}\\
&+C\|u\|_{H^2}\|\p_1\Lambda^{-2}\theta\|_{H^1}
\|\Lambda^{-1}\theta\|_{H^2}\\
\le& C(\|\omega\|_{\dot{H}^{-2}}+\|\Lambda^{-1}\theta\|_{H^2})\\
&\times(\|u\|_{H^2}^2+\|\p_1\theta\|_{\dot{H}^{-2}}^2+\|\p_1\theta\|_{L^2}^2).
\end{aligned}
\end{equation}
Then there exists a positive constant $C_1$ such that
\begin{equation}\label{end2}
\begin{aligned}
\|\p_1\theta\|_{\dot{H}^{-2}}^2-\ddt(\omega|\p_1\theta)_{\dot{H}^{-2}}
=&(\omega|\p_1\theta)_{\dot{H}^{-2}}-
(\p_1^2\Lambda^{-2}\omega|\omega)_{\dot{H}^{-2}}\\
&+(u\cdot\na \omega|\p_1\theta)_{\dot{H}^{-2}}
+(\p_1(u\cdot\nabla \theta)|\omega)_{\dot{H}^{-2}}\\
\le& C_1\|\omega\|_{\dot{H}^{-2}}^2+\frac{1}{2}
\|\p_1\theta\|_{\dot{H}^{-2}}^2\\
&+\|u\cdot\na \omega\|_{\dot{H}^{-2}}^2+\|\p_1{\rm div}(u\otimes\theta)\|_{\dot{H}^{-2}}^2\\
\le& C_1\|\omega\|_{\dot{H}^{-2}}^2+\frac{1}{2}
\|\p_1\theta\|_{\dot{H}^{-2}}^2\\
&+\|u\cdot\na u\|_{\dot{H}^{-1}}^2
+\|u\otimes\theta\|_{L^2}^2\\
\le& C_1\|\omega\|_{\dot{H}^{-2}}^2+\frac{1}{2}
\|\p_1\theta\|_{\dot{H}^{-2}}^2\\
&+\|u\otimes u\|_{L^2}^2
+\|u\|_{L^\infty}^2\|\theta\|_{L^2}^2\\
\le& C_1\|\omega\|_{\dot{H}^{-2}}^2+\frac{1}{2}
\|\p_1\theta\|_{\dot{H}^{-2}}^2\\
&+\|u\|_{L^2}^2\|\omega\|_{L^2}^2
+\|u\|_{L^\infty}^2\|\theta\|_{L^2}^2.
\end{aligned}
\end{equation}
Multiplying (\ref{end1}) by $2C_1$, and adding to (\ref{end2}), we can get
\begin{equation}\label{end3}
\begin{aligned}
&\ \ \ \ddt\{C_1(\|\omega\|_{\dot{H}^{-2}}^2+\|\theta\|_{\dot{H}^{-1}}^2)
-(\omega|\p_1\theta)_{\dot{H}^{-2}}\}\\
&+C_1\|\omega\|_{\dot{H}^{-2}}^2+\frac{1}{2}
\|\p_1\theta\|_{\dot{H}^{-2}}^2\\
\le& (\|\omega\|_{\dot{H}^{-2}}+\|\Lambda^{-1}\theta\|_{H^2}
+\|(\theta,\omega)\|_{L^2}^2)\\
&\times(\|u\|_{H^2}^2+
\|\p_1\theta\|_{\dot{H}^{-2}}^2+\|\p_1\theta\|_{L^2}^2).
\end{aligned}
\end{equation}
Integrating (\ref{end3}) in time, and using
$$2C_1(\|\omega\|_{\dot{H}^{-2}}^2+\|\theta\|_{\dot{H}^{-1}}^2)
-(\omega|\p_1\theta)_{\dot{H}^{-2}}\thickapprox \|\omega\|_{\dot{H}^{-2}}^2+\|\theta\|_{\dot{H}^{-1}}^2,$$
 we have
\begin{equation}\label{300}
\begin{aligned}
&\ \ \|\omega\|_{L^\infty_T(\dot{H}^{-2})}^2
+\|\theta\|_{L^\infty_T(\dot{H}^{-1})}^2
+\int_0^T(\|\omega\|_{\dot{H}^{-2}}^2
+\|\p_1\theta\|_{\dot{H}^{-2}}^2)dt\\
\le& C(\|\omega_0\|_{\dot{H}^{-2}}^2+\|\theta_0\|_{\dot{H}^{-1}}^2)
+\int_0^T(\|\omega\|_{\dot{H}^{-2}}+\|\Lambda^{-1}\theta\|_{H^2}
+\|(\theta,\omega)\|_{L^2}^2)\\
&\times(\|u\|_{H^2}^2+
\|\p_1\theta\|_{\dot{H}^{-2}}^2+\|\p_1\theta\|_{L^2}^2)dt.
\end{aligned}
\end{equation}
Using H\"{o}lder's inequality, we complete the proof of Lemma \ref{ll100}.
\end{proof}
\vskip .2in
\section{ Proof of Theorem \ref{t1.1}}
\label{s7}
\vskip .3in
In this section, we give the proof of Theorem \ref{t1.1}. Local well-posedness can be proved by using standard method.  Here we only show the global a priori bound. Combining with (\ref{200}) and Lemma \ref{ll100}, we can get
\begin{equation}\label{end}
\begin{aligned}
A(T)+A_1(T)
\le& C_2A(0)+C_2A(T)\{A(T)+A_1(T)+A(T)^\frac{3}{2}+A(T)A_1(T)\}\\
&+C_2(A(T)^\frac{5}{3}+A_1(T))\{A(T)^\frac{1}{2}+A_1(T)+A(T)^\frac{1}{4}
A_1(T)^\frac{3}{4}\}.
\end{aligned}
\end{equation}
Denote
$$\bar{T}\stackrel{\rm def}{=}\{T\in (0,T^\star):\ A(T)+A_1(T)\le 4C_2 A(0)\},$$
where $T^\star>0$ is the maximal existence time of the local solution. Assume $\bar{T}<T^\star$. Thanks to (\ref{small}), we can get from (\ref{end}) that
$$A(T)+A_1(T)\le 2C_2A(0),$$
which yields a contradiction with $\bar{T}<T^\star$ by the continuous arguments. Thus we can get
$\bar{T}=T^\star$, which leads to the desired result.

\vskip .4in
\section*{Acknowledgements}
We thank Dr. Zhuan Ye for the helpful comments. We thank the anonymous referees for the careful reading and helpful comments.
This work was supported by the NSF of the Jiangsu Higher Education Institutions of China
(18KJB110018), the NSF of Jiangsu Province BK20180721.

\vskip .4in
\appendix
\section{}
\begin{proof} [Proof of (\ref{App1})]
We only give the estimate of $K_2$, since the estimate of $K_1$
is similar. Integrating by parts, we have
\begin{equation*}
\begin{aligned}
K_2=&\sum_{2\le \alpha\le s+1}C_{s+1}^\alpha\int_0^T\int
\p_2\theta\  \p_2^{\alpha-1}\p_1 u_1\  \p_2^{s+2-\alpha}\theta\  \p_2^{s+1}\theta dxdt\\
=&-\sum_{2\le \alpha\le s+1}C_{s+1}^\alpha\int_0^T\int
\p_1\p_2\theta\  \p_2^{\alpha-1} u_1\  \p_2^{s+2-\alpha}\theta\  \p_2^{s+1}\theta dxdt\\
&-\sum_{2\le \alpha\le s+1}C_{s+1}^\alpha\int_0^T\int
\p_2\theta\  \p_2^{\alpha-1} u_1\  \p_1\p_2^{s+2-\alpha}\theta\  \p_2^{s+1}\theta dxdt\\
&-\sum_{2\le \alpha\le s+1}C_{s+1}^\alpha\int_0^T\int
\p_2\theta\  \p_2^{\alpha-1} u_1\  \p_2^{s+2-\alpha}\theta\  \p_1\p_2^{s+1}\theta dxdt\\
=& \widetilde{K_1}+\widetilde{K_2}+\widetilde{K_3}.
\end{aligned}
\end{equation*}
For $\widetilde{K_1}$, we have
\begin{equation*}
\begin{aligned}
\widetilde{K_1}
=&-(\sum_{2\le \alpha\le 3}+\sum_{4\le \alpha\le s+1})
C_{s+1}^\alpha\int_0^T\int
\p_1\p_2\theta\  \p_2^{\alpha-1} u_1\  \p_2^{s+2-\alpha}\theta\  \p_2^{s+1}\theta dxdt\\
=&\widetilde{K_{11}}+\widetilde{K_{12}}.
\end{aligned}
\end{equation*}
It follows by  H\"{o}lder's inequality and $\|f\|_{L^\infty}
\le\ C\|f\|_{H^2}$ that
\begin{equation*}
\begin{aligned}
\widetilde{K_{11}}\le&\ C\sum_{2\le \alpha\le 3}
\|\p_1\p_2\theta\|_{L^2_T(L^\infty)}\|\p_2^{\alpha-1} u_1\|_{L^2_T(L^\infty)}  \|\p_2^{s+2-\alpha}\theta\|_{L^\infty_T(L^2)}
\|\p_2^{s+1}\theta\|_{L^\infty_T(L^2)}\\
\le&\ C
\|\p_1\theta\|_{L^2_T(H^3)}\| u_1\|_{L^2_T(H^4)}  \|\theta\|_{L^\infty_T(H^{s+1})}^2\\
\le&\ CE(T)E_1(T)
\end{aligned}
\end{equation*}
and
\begin{equation*}
\begin{aligned}
\widetilde{K_{12}}
\le&\ C\sum_{4\le \alpha\le s+1}\|\p_1\p_2\theta\|_{L^2_T(L^\infty)}\|\p_2^{\alpha-1} u_1\|_{L^2_T(L^2)}  \|\p_2^{s+2-\alpha}\theta\|_{L^\infty_T(L^\infty)}
\|\p_2^{s+1}\theta\|_{L^\infty_T(L^2)}\\
\le&\ C\|\p_1\theta\|_{L^2_T(H^3)}\| u_1\|_{L^2_T(H^s)}  \|\theta\|_{L^\infty_T(H^{s+1})}^2\\
\le&\ CE(T)E_1(T).
\end{aligned}
\end{equation*}
Thus, we get
\begin{equation}\label{aa}
\widetilde{K_1}
\le\ CE(T)E_1(T).
\end{equation}
Along the similar arguments can yield the estimate of $\widetilde{K_2}$:
$$\widetilde{K_2}
\le\ CE(T)E_1(T).$$
For $\widetilde{K_3}$, integrating by parts again, one can deduce that
\begin{equation}\label{aa2}
\begin{aligned}
\widetilde{K_3}=& \sum_{2\le \alpha\le s+1}C_{s+1}^\alpha\int_0^T\int
\p_2^2\theta\  \p_2^{\alpha-1} u_1\  \p_2^{s+2-\alpha}\theta\  \p_1\p_2^s\theta dxdt\\
& +\sum_{2\le \alpha\le s+1}C_{s+1}^\alpha\int_0^T\int
\p_2\theta\  \p_2^{\alpha} u_1\  \p_2^{s+2-\alpha}\theta\  \p_1\p_2^{s}\theta dxdt\\
& +\sum_{2\le \alpha\le s+1}C_{s+1}^\alpha\int_0^T\int
\p_2\theta\  \p_2^{\alpha-1} u_1\  \p_2^{s+3-\alpha}\theta\  \p_1\p_2^{s}\theta dxdt\\
\end{aligned}
\end{equation}
Like the previous way yielding (\ref{aa}),  the three parts
on the right hand side of (\ref{aa2}) can also be bounded by
$CE(T)E_1(T)$. Collecting the above estimates can yield the desired bound of $K_2$.
\end{proof}

\end{document}